\documentclass[12pt]{article}
\usepackage{mathrsfs}

\usepackage{amssymb, amsthm, amsfonts, amsxtra, amsmath}
\usepackage{latexsym}
\usepackage{mathabx}

\usepackage{parskip}
\makeatletter 
\def\thm@space@setup{
 \thm@preskip=\parskip \thm@postskip=0pt
}
\def\th@remark{
  \thm@headfont{\itshape}
  \normalfont
  \thm@preskip\parskip \thm@postskip=0pt
}
\makeatother

\usepackage[]{hyperref}

\numberwithin{equation}{section}

\DeclareMathOperator{\id}{id} \DeclareMathOperator{\ext}{\mathrm{ext}}  
\DeclareMathOperator{\fin}{\mathrm{fin}} \DeclareMathOperator{\Pol}{\mathrm{Pol}} \DeclareMathOperator{\End}{\mathrm{End}}

 \newcommand{\g}{\mathfrak{g}} \newcommand{\h}{\mathfrak{h}} \newcommand{\kk}{\mathfrak{k}}  
\newcommand{\n}{\mathfrak{n}}   \newcommand{\bb}{\mathfrak{b}}
 \newcommand{\ssl}{\mathfrak{sl}}   \newcommand{\G}{\mathbb{G}}
  \newcommand{\C}{\mathbb{C}} \newcommand{\R}{\mathbb{R}} \newcommand{\Z}{\mathbb{Z}} \newcommand{\N}{\mathbb{Z}_+}
   \newcommand{\T}{\mathbb{T}}  
 \newcommand{\Hsp}{\mathscr{H}}  
\newcommand{\bqn}[3]{\left\lbrack \begin{array}{c} \!#1\! \\ \!#2\! \end{array}\right\rbrack_{#3}} 
\newcommand{\QQ}{\mathcal{Q}}
\newcommand{\RR}{\mathcal{R}}
     \newcommand{\wt}{\mathrm{wt}}
 \newcommand{\du}[1]{#1^{\vee}}  \newcommand{\K}{\mathbb{K}} 
\newcommand{\U}{\mathcal{U}}

\newcommand{\CU}{\check{U}}
\newcommand{\SU}{\mathrm{SU}}
\newcommand{\SSS}{\mathcal{S}}

\newtheorem{Theorem}{Theorem}[section] \newtheorem{Lem}[Theorem]{Lemma} \newtheorem{Prop}[Theorem]{Proposition} \newtheorem{Cor}[Theorem]{Corollary}

\theoremstyle{definition} \newtheorem{Def}[Theorem]{Definition} \newtheorem{Rem}[Theorem]{Remark}

\newcommand\bp{\begin{proof}}
\newcommand\ep{\end{proof}}

\begin{document}

\title{Quantum flag manifolds as quotients of degenerate quantized universal enveloping algebras}

\author{Kenny De Commer\thanks{This work is part of the project supported by the NCN-grant 2012/06/M/ST1/00169. Department of mathematics, Vrije Universiteit Brussel, VUB, B-1050 Brussels, Belgium, email: {\tt Kenny.De.Commer@vub.ac.be}}
\and Sergey Neshveyev\thanks{Research supported by ERC Grant no.~307663. Department of Mathematics, University of Oslo, P.O. Box 1053 Blindern, NO-0316 Oslo, Norway, email: {\tt sergeyn@math.uio.no}}}

\date{}

\maketitle

\begin{abstract} \noindent Let $\mathfrak{g}$ be a semi-simple Lie algebra with fixed root system, and $U_q(\mathfrak{g})$ the quantization of its universal enveloping algebra. Let $\mathcal{S}$ be a subset of the simple roots of~$\g$. We show that the defining relations for $U_q(\mathfrak{g})$ can be slightly modified in such a way that the resulting algebra $U_q(\g;\SSS)$ allows a homomorphism onto (an extension of) the algebra $\mathrm{Pol}(\mathbb{G}_q/\K_{\SSS,q})$ of functions on the quantum flag manifold~$\G_q/\K_{\SSS,q}$ corresponding to $\SSS$. Moreover, this homomorphism is equivariant with respect to a natural adjoint action of $U_q(\g)$ on $U_q(\g;\SSS)$ and the standard action of $U_q(\g)$ on $\mathrm{Pol}(\mathbb{G}_q/\K_{\SSS,q})$.
\end{abstract}

\emph{Keywords}: Quantum universal enveloping algebras, quantum flag manifolds, compact quantum groups

AMS 2010 \emph{Mathematics subject classification}: 17B37; 20G42; 81R50

\section*{Introduction}

Let $\G$ be a semi-simple simply connected compact Lie group with Lie algebra $\g$. As was observed by Drinfel'd \cite{Dri1}, see also \cite{Gav1}, the Drinfel'd-Jimbo quantization $U_q(\g)$ of the universal enveloping algebra $U(\g)$ can also be seen as a quantization of the function algebra on the Poisson-Lie dual $\G^*$ of $\G$. In this way, $U_q(\g)$ and its collection of irreducible representations can be interpreted as a quantization of $\G^*$ and its collection of symplectic leaves, an interpretation connected with Kirillov's orbit method. In a similar spirit, a flag manifold $\G/\K$, being identified with some coadjoint orbit in $\g^*$, is sometimes viewed as a limit of the irreducible quotients of $U(\g)$ when the dimension tends to infinity, an idea made precise in \cite{Lan1,Rie1}.

In \cite{EEM1}, see also \cite{Mud1}, flag manifolds were considered with a Poisson structure obtained from the associated dynamical $r$-matrix together with a character on the Lie algebra of the stabilizer. In case this data satisfied a certain regularity condition, it was shown (in the formal deformation setting) that the quantization of this flag manifold could be constructed as a quotient by the kernel of a representation on a suitable generalized Verma module.

We contribute to this circle of ideas by showing that the relations of $U_q(\g)$ can be modified as to produce an algebra $U_q(\g;\SSS)$ which projects onto (an extension of) the algebra $\Pol(\G_q/\K_{\SSS,q})$ of functions on the quantum flag manifold $\G_q/\K_{\SSS,q}$ associated with a set $\SSS$ of simple roots. The algebras $U_q(\g;\SSS)$ are particular examples of algebras studied in~\cite{DeC1}. They can be obtained by an appropriate rescaling of the generators of $U_q(\g)$ and sending the parameters to $0$. The key observation is that this rescaling is invariant with respect to the natural adjoint action of $U_q(\g)$ on itself. This endows the limit algebra $U_q(\g;\SSS)$ with a natural action of $U_q(\g)$. It then suffices to define our homomorphism on the Cartan subalgebra of the limit algebra, and to use equivariance to extend it to the whole algebra.

A part of our main theorem could be easily deduced from known results. Namely, in the full flag manifold case a closely related isomorphism of $U_q(\bb^\pm)$ onto (a localization of) a subalgebra of $\Pol(\G_q)$ was constructed by De Concini and Procesi in~\cite[Section~2]{DCP} (and a similar, but different, isomorphism earlier in~\cite[Theorem~4.6]{DCL}), see also~\cite[Theorem~3.7]{Ya} and~\cite[Theorem 17]{KOY1} for recent different proofs. Our proof is, however, independent of these results, and in the full flag manifold case it is quite short anyways.

As we do not work in the formal deformation setting, and as our methods are more direct, we make no precise connection with the work of \cite{EEM1}. The extension of our work to the case of quantum homogeneous spaces coming from non-standard Poisson structures would certainly be interesting, especially in connection with real structures~\cite{DeC1}, but will be left for a future occasion.

The article consists of two sections. In the \emph{first section}, we recall the structure and representation theory of the quantized universal enveloping algebras $U_q(\g)$ and their duals, the quantum function algebras $\Pol(\G_q)$. In the \emph{second section}, we then derive our main result. We construct particular elements in the algebra $\Pol(\G_q/\K_{\SSS,q})$, and show that they display a $U_q(\g)$-like behavior. Upon passing to a slightly larger version of $\Pol(\G_q/\K_{\SSS,q})$, we show that the latter algebra can indeed be realized as a quotient of a degenerate version of~$U_q(\g)$. 

\emph{Acknowledgements}:  Work on this article was begun when the first author was affiliated with the Laboratoire de math\'ematiques AGM of the University of Cergy-Pontoise. He would like to thank its members for the pleasant and accommodating working environment.

\section{Preliminaries}

We fix a complex semi-simple Lie algebra $\g$ of rank $l$, pick a fixed Cartan subalgebra $\h$ and write the Cartan decomposition as $\g = \n^-\oplus \h \oplus \n^+$. We write $\bb = \bb^+ = \h\oplus \n^+$, and $\bb^- = \n^-\oplus \h$. We label the set $\Phi^+$ of simple positive roots by the set $I=\{1,\ldots,l\}$, and write $\Phi^+ = \{\alpha_r\mid r\in I\}$. We denote by $Q^+$ its $\Z_+$-span, by $Q$ its $\Z$-span (the root lattice) and by $\h^*_{\R}$ its $\R$-span. We let $(\,,\,)$ be any positive definite form on $\h_{\R}^*$ for which $A = (a_{rs})_{r,s\in I} =  \left((\du{\alpha}_r,\alpha_s)\right)_{r,s\in I}$ is the Cartan matrix of $\g$, where $\du{\alpha} = \frac{2}{(\alpha,\alpha)}\alpha$ for $\alpha\in \Phi^+$. We write $\Phi_d^+ = \{\omega_r\mid r\in I\}$ for the set of fundamental weights in $\h_{\R}^*$, so $(\omega_r,\du{\alpha}_s) = \delta_{rs}$. The $\Z_+$-span of $\Phi_d^+$ is denoted $P^+$, its $\Z$-span $P$ (the weight lattice).

We further fix a deformation parameter $0<q<1$, and write $q_r = q^{\frac{(\alpha_r,\alpha_r)}{2}}$ for $r\in \Phi^+$. We remark that if in what follows we ignore the $^*$-structure, our main result can be easily extended to the case of complex $q\ne0$ such that $q_r$ are not roots of unity.

For $m\geq n\geq 0$, we write \[\bqn{m}{n}{r} = q_r^{n(n-m)}\prod_{k=1}^n \frac{(1-q_r^{2m-2k+2})}{(1-q_r^{2k})}.\]

\subsection{Quantized universal enveloping algebras}

The following treatment of the algebraic structure of quantized universal enveloping (QUE) algebras is mainly based on \cite[Sections 2-4]{JoL1} (note that our $q$ is their $q^2$.)

\begin{Def} We define $\tilde{U}_q(\bb)$ to be the universal unital algebra generated by elements $E_r$, $r\in I$, as well as elements $L_{\omega}$, $\omega\in P$, such that for all $r \in I$ and $\omega,\chi\in P$, we have \begin{enumerate}
\item[1)] $L_{\omega}$ is invertible and $L_{\chi}L_{\omega}^{-1} = L_{\chi-\omega}$,
\item[2)] $L_{\omega}E_rL_{\omega}^{-1}   = q^{\frac{(\omega,\alpha_r)}{2}} E_r$.
\end{enumerate}
\end{Def}

There exists a unique Hopf algebra structure $(\tilde{U}_q(\bb),\Delta)$ on $\tilde{U}_q(\bb)$ such that $\Delta(L_{\omega}) = L_{\omega}\otimes L_{\omega}$ and \[\Delta(E_r) = E_r\otimes L_{\alpha_r}+ L_{\alpha_r}^{-1}\otimes E_r.\]
The co-unit $\varepsilon$ is determined by $\varepsilon(E_r)=0$ for all $r$ and $\varepsilon(L_{\omega})=1$ for all $\omega$, while the antipode $S$ satisfies $S(L_{\omega})=L_{\omega}^{-1}$ and $S(E_r) =-q_rE_r$.

As $\tilde{U}_q(\bb)$ is a Hopf algebra, it carries a right adjoint action \[x\lhd y = S(y_{(1)})xy_{(2)},\] where we have used the Sweedler notation $\Delta(x) = x_{(1)}\otimes x_{(2)}$. In this way, $\tilde{U}_q(\bb)$ becomes a right $\tilde{U}_q(\bb)$-module algebra over itself.

\begin{Def} We define $U_q(\bb)$ to be the quotient algebra of $\tilde{U}_q(\bb)$ by the relations \begin{equation}\label{EqAdNil} L_{-4\omega_s}\lhd (E_sE_r^{1-a_{rs}})=0,\qquad \textrm{for all }r\neq s.\end{equation}
\end{Def}

\begin{Lem} Condition \eqref{EqAdNil} is equivalent to the quantum Serre relations
\[ \sum_{k=0}^{1-a_{rs}}  (-1)^k \bqn{1-a_{rs}}{k}{r} E_r^kE_sE_r^{1-a_{rs}-k} = 0 \qquad \textrm{for all }r\neq s,\] and the coproduct on $\tilde{U}_q(\bb)$ descends to a Hopf algebra coproduct on $U_q(\bb)$.
\end{Lem}
\begin{proof} See \cite[Lemmas 4.5-4.9]{JoL1}.
\end{proof}

In the following, we will denote $U_q(\bb^+)$ for $U_q(\bb)$ with the above coproduct, and $U_q(\bb^-)$ for $U_q(\bb)$ with the opposite coproduct. The generators of $U_q(\bb^-)$ will then be written $\{F_r,L_{\omega}'\}$.

\begin{Def} We define $U_q(\g)$ to be the unital algebra with the universal property that it contains and is generated by $U_q(\bb^+)$ and $U_q(\bb^-)$ as subalgebras, and with $L_{\omega}^{-1} = L_{\omega}'$ and \[\lbrack E_r,F_s\rbrack = \delta_{rs} \frac{L_{\alpha_r}^2-L_{\alpha_r}^{-2}}{q_r-q_r^{-1}}.\]
\end{Def}

It is known that the homomorphisms of $U_q(\bb^+)$ and $U_q(\bb^{-})$ into $U_q(\g)$ are faithful, see e.g.~\cite[Lemma 4.8]{JoL1}, \cite[p. 170, Remark 6]{KS1}.

Finally, we define a Hopf $^*$-algebra structure $(U_q(\g),\Delta)$ on $U_q(\g)$ such that $\Delta$ coincides with the above coproduct on $U_q(\bb^{+})$ and $U_q(\bb^{-})$, and with the $^*$-structure given by $E_r^* = F_r$ and $L_{\omega}^* = L_{\omega}$.

\begin{Rem} The above version of $U_q(\g)$ is sometimes referred to as the `simply connected version' of the QUE algebra, since the Cartan subalgebra is generated by the weight lattice (cf. \cite[Remark 4 following Definition-Proposition 9.1.1]{CP}). One can also work with the root lattice in the above definition, in which case one obtains a smaller Hopf $^*$-algebra which we will denote by $\CU_q(\g)$.
\end{Rem}

The antipode on $U_q(\g)$ is given by
$$
S(L_\omega)=L_{-\omega},\ \ S(E_r)=-q_rE_r,\ \ S(F_r)=-q_r^{-1}F_r.
$$
A simple computation shows that if we let $\rho = \sum_i\omega_i$, then $L_{-4\rho}^{-1}xL_{-4\rho} = S^2(x)$ for all $x\in U_q(\g)$.

It will also be convenient to use the unitary antipode, defined by $R(x)=L_{-2\rho}S(x)L_{2\rho}$, so $R$ is an involutive $^*$-anti-automorphism of $U_q(\g)$ such that
$$
R(L_\omega)=L_{-\omega},\  \ R(E_r)=-E_r,\ \ R(F_r)=-F_r.
$$

\subsection{Representation theory of quantized enveloping algebras}

We now recall basic facts about the representation theory of $U_q(\g)$, see e.g.~\cite[Chapter~10]{CP}.

By a \emph{type $I$ representation} of $U_q(\g)$ we will mean a unital $^*$-representation of $U_q(\g)$ on a finite-dimensional Hilbert space $V$ such that the operators $L_{\omega}$ are positive.  For such a~$V$, we say that $\xi\in V$ is a vector of weight $\lambda$ if $\xi$ is non-zero and $L_{\omega}\xi = q^{\frac{1}{2}(\omega,\lambda)}\xi$ for all $\omega$. We denote by $V(\lambda)$ the space of vectors of weight $\lambda$. If $\xi$ is a weight vector, we will denote its weight by $\wt(\xi)$. Any type $I$ representation is a direct sum of its weight subspaces $V(\lambda)$.
The type $I$-representations of $U_q(\g)$ and $\CU_q(\g)$ coincide.

We denote by $\bar V$ the complex conjugate Hilbert space equipped with the representation defined by $x\bar\xi=\overline{R(x^*)\xi}$, where $R$ is the unitary antipode.

We say that $\xi$ is a highest weight vector of weight $\lambda$ if it is a vector of weight $\lambda$ which is annihilated by all $E_r$. A type $I$ representation is irreducible if and only if it is generated by a highest weight vector. Moreover, $\lambda$ appears as a highest weight for some irreducible type $I$-representation if and only if $\lambda\in P^+$. If $\lambda$ is a positive integral weight, we write $V_{\lambda}$ for the associated irreducible module. Then $\bar V_{\lambda}  \cong V_{-w_0\lambda}$, where $w_0$ is the longest element in the Weyl group of $\g$. We will once and for all choose unit norm highest weight vectors~$h_{\lambda}$ for each $V_{\lambda}$. Then $\bar h_{\lambda}$ will be a lowest weight vector in $\bar V_{\lambda}$ of weight $-\lambda$.

Let \[\U_q(\g) = \prod_{\omega\in P^+} B(V_{\omega}),\] and consider $U_q(\g)\subseteq \U_q(\g)$ in the natural way. We can equip $\U_q(\g)$ with a `coproduct' $\Delta$ into \[\U_q(\g)\hat{\otimes}\U_q(\g) =\prod_{\omega,\lambda\in P^+} \left(B(V_{\omega})\otimes B(V_{\lambda})\right)\] such that its restriction to $U_q(\g)$ becomes the ordinary coproduct. It is clear that any type $I$ representation of $U_q(\g)$ can be extended uniquely to $\U_q(\g)$, and similarly tensor products of type $I$-representations for $U_q(\g)^{\otimes n}$ can be extended uniquely to $\U_q(\g)^{\hat{\otimes}n}$.

With our conventions, the universal $R$-matrix $\RR$ is then an element of $\U_q(\g)\hat{\otimes}\U_q(\g)$, uniquely determined by the conditions that
$\RR\Delta(x)\RR^{-1} = \Delta^{\mathrm{op}}(x)$ for all $x\in U_q(\g)$ and \[\RR(h_{\lambda}\otimes \bar h_{\mu}) = q^{-(\lambda,\mu)}(h_{\lambda}\otimes \bar h_{\mu})\] on $V_{\lambda}\otimes\bar V_{\mu}$. An explicit formula for $\RR$ can be found e.g.~in \cite[Theorem~8.3.9]{CP}. The only thing we will need is that the $R$-matrix has the form $\RR=\QQ\tilde\RR$, where
\begin{equation}\label{ermatrix}
\QQ(\xi\otimes\eta)=q^{(\wt(\xi),\wt(\eta))}\xi\otimes\eta\ \ \text{and}\ \ \tilde\RR=1+\sum_{r\in I} q_r^{-1}(q_r-q_r^{-1})L_{\alpha_r}E_r\otimes L_{-\alpha_r}F_r+\dots,
\end{equation}
and where the additional terms of $\tilde\RR$ map $\xi\otimes\eta\in V\otimes V'$ into $V(\wt(\xi)+\alpha)\otimes V'(\wt(\eta)-\alpha)$ with $\alpha\in Q^+\setminus(\Phi^+\cup\{0\})$,.

\subsection{Quantized semi-simple compact Lie groups}

We now turn to a Hopf algebra dual of $U_q(\g)$, see e.g.~\cite{KoS1,NeT1}.

The unital $^*$-algebra $\Pol(\G_q)$ is defined as the subspace of $U_q(\g)^*$ spanned by the linear functionals  of the form \[U(\xi,\eta)\colon x\mapsto \langle \xi,x\eta\rangle,\] where $\xi,\eta$ belong to a type $I$ representation $V$ of $U_q(\g)$ (clearly, we assume that Hermitian scalar products are linear in the second variable). It is a Hopf $^*$-algebra with the product/coproduct dual to the coproduct/product of $U_q(\g)$, and with the $^*$-operation given by
$\omega^*(x)=\overline{\omega(S(x)^*)}$. Recalling that the module structure on $\bar V$ is defined using the unitary antipode, we can also write
\begin{equation}\label{estar}
U(\xi,\eta)^*=q^{-(\rho,\wt(\xi)-\wt(\eta))}U(\bar\xi,\bar\eta).
\end{equation}

The symbol $\G_q$ should be seen as a quantization of the simply connected compact Lie group integrating the compact form of $\g$.

The $^*$-algebra $\Pol(\G_q)$ is a $U_q(\g)$-bimodule $^*$-algebra, by means of the module structures \[x\rhd U(\xi,\eta)\lhd y = U(y^*\xi,x\eta).\]

\begin{Lem}\label{LemSwitch} Let $V_i$ be type $I$ representations, and $\xi_i,\eta_i \in V_i$, $i=1,2$. Then
\begin{eqnarray*} U(\xi_1,\eta_1)U(\xi_2,\eta_2) &=& U\left(\RR_{21}(\xi_2\otimes \xi_1),\RR^{-1}(\eta_2\otimes \eta_1)\right)\\
&=& U\left(\RR^{-1}(\xi_2\otimes \xi_1),\RR_{21}(\eta_2\otimes \eta_1)\right).
\end{eqnarray*}
\end{Lem}

\begin{proof} Straightforward from the fact that $\RR$ and $\RR_{21}^{-1}$ flip the coproduct, and the fact that $\RR^* = \RR_{21}$ (cf. \cite[Example 2.6.4]{NeT1}).
\end{proof}

The $^*$-algebra $\Pol(\G_q)$ has as well an interesting and tractable representation theory. Consider first the case where $\g = \ssl(2,\C)$, in which case we let $I$ consist of the empty symbol, and assume that $(\alpha,\alpha)=2$. We then write $\G_q = \SU_q(2)$. As usual, we label weights by half-integers. Write $h_{1/2} \in V_{1/2}$ for the highest weight vector and $h_{-1/2} = Fh_{1/2}$. Then the matrix \[\begin{pmatrix} a & c\\ b & d\end{pmatrix} = \begin{pmatrix} U(h_{-1/2},h_{-1/2})& U(h_{-1/2},h_{1/2})\\ U(h_{1/2},h_{-1/2})& U(h_{1/2},h_{1/2})\end{pmatrix}\] is unitary in $\Pol(\SU_q(2))\otimes M_2(\C)$, with $c=-qb^*$ and $d = a^*$. Moreover, these relations together with the unitarity of the above matrix provide universal relations for $\Pol(\SU_q(2))$. It is then easy to find all irreducible representations of $\Pol(\SU_q(2))$. There is one family of one-dimensional representations $\theta_z$ for $z\in \T=\{w\in \C\mid |w|=1\}$, given by \[\theta_z\begin{pmatrix} a & c \\ b & d\end{pmatrix} = \begin{pmatrix} z & 0\\ 0 & \bar{z}\end{pmatrix}.\] Apart from this, there is only one other one-parameter family of irreducible representations $\theta\otimes \theta_z$, where $\theta$ is the representation of $\Pol(\SU_q(2))$ on $l^2(\Z_+)$ by the operators \[ae_n = (1-q^{2n})^{1/2}e_{n-1},\qquad be_n = q^{n}e_n.\]

Consider now a general $\Pol(\G_q)$. Then $U_q(\g)$ contains a copy of the Hopf $^*$-algebra $\CU_{q_r}(\ssl(2,\C))$, generated by the elements $E_r,F_r$ and $L_{\alpha_r}^{\pm 1}$. It follows that one has a natural Hopf $^*$-algebra homomorphism \[\pi_r\colon \Pol(\G_q)\rightarrow \Pol(\SU_{q_r}(2)).\] Denote by $\theta_r$ the composition of $\pi_r$ with the irreducible representation $\theta$ of $\Pol(\SU_{q_r}(2))$ on $l^2(\Z_+)$. Let $s_i$ denote the reflections around simple roots in $\h_{\R}^*$, and let \[w = s_{i_1}\cdots s_{i_{l(w)}}\] be a reduced expression for an element $w$ in the Weyl group $W$ of $\g$, where $l(w)$ is the length of $w$. Then we can form the representation \[\theta_w = (\theta_{i_1}\otimes \cdots \otimes \theta_{i_{l(w)}})\Delta^{(l(w))}\] on $l^2(\Z_+)^{\otimes l(w)}$. The equivalence class of this representation is independent of the presentation of $w$. One, moreover, has a family of one-dimensional representations $\theta_z$ labeled by $z=(z_1,\ldots,z_l)\in \T^l$ and determined by \[\theta_z(U(\xi,\eta))= \langle\xi,\eta\rangle \prod_{k\in I} z_k^{(\alpha_k^\vee,\wt(\eta))}.\] Then any irreducible representation of $\Pol(\G_q)$ is equivalent to one of the form $\theta_{w,z}=(\theta_w\otimes \theta_z)\Delta$ on the Hilbert space $\Hsp_{w,z} = \ell^2(\Z_+)^{\otimes l(w)}$, and these representations are all mutually inequivalent.

Soibelman's proof of these results \cite{Soi1,KoS1} is based on an analysis of operators $U(h_{w\lambda},h_\lambda)$, where $h_{w\lambda}$ is a fixed unit vector in the one-dimensional space $V_\lambda(w\lambda)$. For our purposes it will be more convenient to use the operators $U(h_\lambda,h_{w^{-1}\lambda})$. In fact, in our conventions these operators are more natural, since they are of norm one.

\begin{Prop}\label{psoibelman}
Let $w\in W$ be an element of the Weyl group and $\lambda\in P^+$. Then
\begin{enumerate}
\item[{\rm 1)}]  if $\eta\in V_\lambda$ is orthogonal to $U_q(\bb^+)h_{w^{-1}\lambda}$, then $\theta_w(U(h_\lambda,\eta))=0$;
\item[{\rm 2)}] the operator $\theta_w(U(h_\lambda,h_{w^{-1}\lambda}))$ is diagonal with respect to the standard
basis of $\Hsp_w$, all its eigenvalues are nonzero and have modulus $\le1$, and $e_0^{\otimes l(w)}$ is an eigenvector with eigenvalue of modulus $1$; furthermore, if $\lambda$ is regular, then $e_0^{\otimes l(w)}$ is the only eigenvector with eigenvalue of modulus $1$.
\end{enumerate}
\end{Prop}

\begin{proof}
This is similar to the results in \cite{Soi1}.
We will sketch a proof for the reader's convenience. The proof is by induction on $l(w)$. Assume the result is true for $w$ and let's prove it for $ws_i$ such that $l(ws_i)>l(w)$. Take $\eta\in V_\lambda$. Choose an orthonormal basis $\{\xi_j\}_j$ in $V_\lambda$ such that each $\xi_j$ either lies in or is orthogonal to $U_q(\bb^+)h_{s_iw^{-1}\lambda}$. We have
$$
\theta_{ws_i}(U(h_\lambda,\eta))
=\sum_j\theta_w(U(h_\lambda,\xi_j))\otimes\theta_{i}(U(\xi_j,\eta)).
$$
Since $l(s_iw^{-1})>l(w^{-1})$, by \cite[Lemma 4.4.3(v)]{Jo} the $U_q(\bb^+)$-module $U_q(\bb^+)h_{s_iw^{-1}\lambda}$ contains $U_q(\bb^+)h_{w^{-1}\lambda}$ and is a $\CU_{q_i}(\ssl(2,\C))$-module. By the inductive assumption and the definition of $\theta_i$ it follows that if the summand corresponding to an index $j$ in the above expression is nonzero, then $\xi_j\in U_q(\bb^+)h_{w^{-1}\lambda} \subseteq U_q(\bb^+)h_{s_iw^{-1}\lambda}$ and $\eta$ is not orthogonal to $\CU_{q_i}(\ssl(2,\C))\xi_j\subseteq U_q(\bb^+)h_{s_iw^{-1}\lambda}$. This proves (1). By \cite[Lemma 4.4.3(ii),(iii)]{Jo} we also have $E_ih_{w^{-1}\lambda}=0$ and $h_{w^{-1}\lambda}\in \C E_i^mh_{s_iw^{-1}\lambda}$ for $m=(w^{-1}\lambda,\alpha_i^\vee)\ge0$. From this we get
\begin{multline*}
\theta_{ws_i}(U(h_\lambda,h_{s_iw^{-1}\lambda}))
=\theta_w(U(h_\lambda,h_{w^{-1}\lambda}))\otimes \theta_i(U(h_{w^{-1}\lambda},h_{s_iw^{-1}\lambda}))\\
=z\theta_w(U(h_\lambda,h_{w^{-1}\lambda}))\otimes \theta(b)^m,
\end{multline*}
where $z$ is a scalar factor of modulus $1$ and $b$ is the element $U(h_{1/2},h_{-1/2})\in\Pol(\SU_{q_i}(2))$. Using induction this allows one to explicitly compute $\theta_{ws_i}(U(h_\lambda,h_{s_iw^{-1}\lambda}))$ and prove (2).
\end{proof}

\begin{Cor} \label{ccommut}
Assume $w\in W$ and $\lambda\in P^+$. Then for all weight vectors $\xi,\eta$ in a type~$I$ $U_q(\g)$-module $V$ we have
$$
U(\xi,\eta)U(h_\lambda,h_{w^{-1}\lambda})=q^{-(\lambda,\wt(\xi)-w\wt(\eta))}
U(h_\lambda,h_{w^{-1}\lambda})U(\xi,\eta)\mod\ker\theta_w,
$$
$$
U(\xi,\eta)^*U(h_\lambda,h_{w^{-1}\lambda})=q^{(\lambda,\wt(\xi)-w\wt(\eta))}
U(h_\lambda,h_{w^{-1}\lambda})U(\xi,\eta)^*\mod\ker\theta_w.
$$
\end{Cor}

\bp The first identity follows from Lemma~\ref{LemSwitch}, since by \eqref{ermatrix} we have $\RR^{-1}(h_\lambda\otimes\xi)=q^{-(\lambda,\wt(\xi))}h_\lambda\otimes\xi$, while $\RR_{21}(h_{w^{-1}\lambda}\otimes\eta)$ is equal to $q^{(\lambda, w\wt(\eta))}h_{w^{-1}\lambda}\otimes\eta$ plus vectors in $V_\lambda(w^{-1}\lambda-\alpha)\otimes V(\wt(\eta)+\alpha)$ for $\alpha\in Q^+\setminus\{0\}$, which produce matrix elements in the kernel of $\theta_w$. The second equality follows from the first one by recalling that by \eqref{estar} the element $U(\xi,\eta)^*$ coincides with $U(\bar\xi,\bar\eta)$ up to a scalar factor.
\ep

\subsection{Quantum subgroups of quantized semi-simple compact Lie groups}

Let $\SSS\subseteq \Phi^+$ be an arbitrary subset. Write $U_q(\kk_\SSS)$ for the Hopf $^*$-subalgebra of $U_q(\g)$ generated by all $L_{\omega}$ and all $E_r,F_r$ with $r\in \SSS$.

Write $\Pol(\K_{\SSS,q})\subseteq U_q(\kk_\SSS)^*$ for the resulting quotient of $\Pol(\G_q)$, with quotient map $\pi_\SSS$. We write \[\Pol(\G_q/\K_{\SSS,q}) = \{x\in \Pol(\G_q)\mid (\id\otimes \pi_\SSS)\Delta(x) = x\otimes 1\},\] and call $\G_q/\K_{\SSS,q}$ the associated quantum homogeneous space, or more specifically the \emph{quantum flag manifold} associated to $\SSS$. If $\SSS=\emptyset$, we call $\G_q/\T^l$ the \emph{full quantum flag manifold} of $\G_q$. For general $\SSS$, the algebra $\Pol(\G_q/\K_{\SSS,q})$ is a left coideal, in that \[\Delta(\Pol(\G_q/\K_{\SSS,q})) \subseteq \Pol(\G_q)\otimes\Pol(\G_q/\K_{\SSS,q}).\]

\section{Commutation relations in \texorpdfstring{$\Pol(\G_q/\K_{\SSS,q})$}{Pol(Gq/KS,q)}}

In the following we fix $\SSS\subseteq \Phi^+$. Whenever convenient, we will drop the $\SSS$ from the notation.

\subsection{An algebra of functions on the quantum big Schubert cell}

Denote by $W_\SSS$ the subgroup of the Weyl group generated by the simple reflections $s_i$ corresponding to $\alpha_i\in\SSS$. Let $w$ be the shortest element in $w_0W_\SSS$, where $w_0$ denotes the longest element in $W$. It is known that $\theta_w$ is an irreducible faithful representation of $\Pol(\G_q/\K_{\SSS,q})$ \cite[Theorem 5.9]{SD}. Because of this we will drop the notation $\theta_w$ when applied to elements of $\Pol(\G_q/\K_{\SSS,q})$ whenever it is convenient.

\begin{Def}\label{dk}
For $\lambda\in P^+$ define $k_{-4\lambda}\in\Pol(\G_q/\K_{\SSS,q})$ as the unique element such that
$$
k_{-4\lambda}=U(h_\lambda,h_{w^{-1}\lambda})^*U(h_\lambda,h_{w^{-1}\lambda})\mod\ker\theta_w.
$$
\end{Def}

Of course, we have to check that such elements exist. First of all, by \eqref{estar} we have
\begin{align*}
U(h_\lambda,h_{w^{-1}\lambda})^*U(h_\lambda,h_{w^{-1}\lambda})
&=q^{-(\rho,\lambda-w^{-1}\lambda)}U(\bar h_\lambda\otimes h_\lambda,
\bar h_{w^{-1}\lambda}\otimes h_{w^{-1}\lambda})\\
&=q^{-(\rho-w\rho,\lambda)}U(\bar h_\lambda\otimes h_\lambda,\bar h_{w^{-1}\lambda}\otimes h_{w^{-1}\lambda}).
\end{align*}
Now, consider the lowest weight $U_q(\kk_\SSS)$-module $V=U_q(\kk_\SSS)h_{w_0\lambda}\subseteq V_\lambda$. Let $w_{\SSS,0}$ be the longest element in $W_\SSS$.  Then  $w_{\SSS,0}w^{-1}\lambda=w_0\lambda$, so the highest weight of the $U_q(\kk_\SSS)$-module~$V$ is $w^{-1}\lambda$. Hence $h_{w^{-1}\lambda}$ is a highest weight vector of this module. The $U_q(\kk_\SSS)$-module $\bar V\otimes V$ contains a $U_q(\kk_\SSS)$-invariant vector $v_\lambda\ne0$, unique up to a scalar factor. We normalize this vector $v_\lambda$ so that
$\langle v_\lambda,\bar h_{w^{-1}\lambda}\otimes h_{w^{-1}\lambda}\rangle=1$. Since $V$ contains only vectors of weights not larger than $w^{-1}\lambda$, and the vector $v_\lambda$ is of weight zero, we have
\begin{equation} \label{einvvect}
v_\lambda=\bar h_{w^{-1}\lambda}\otimes h_{w^{-1}\lambda}\mod\sum_{\alpha\in Q^+\setminus\{0\}}\overline{ V_\lambda(w^{-1}\lambda-\alpha)}\otimes V_\lambda(w^{-1}\lambda-\alpha).
\end{equation}
By Proposition~\ref{psoibelman}(1) we conclude that
$$
U(\bar h_\lambda\otimes h_\lambda,\bar h_{w^{-1}\lambda}\otimes h_{w^{-1}\lambda})=U(\bar h_\lambda\otimes h_\lambda,v_\lambda) \mod\ker\theta_w.
$$
Therefore
\begin{equation} \label{ek}
k_{-4\lambda}=q^{-(\rho-w\rho,\lambda)}U(\bar h_\lambda\otimes h_\lambda,v_\lambda).
\end{equation}

\begin{Rem}\label{rinvvect}
A byproduct of this argument is that, up to a normalization, $k_{-4\lambda}$ is equal to $$p_\SSS\rhd (U(h_\lambda,h_{w_0\lambda})^*U(h_\lambda,h_{w_0\lambda})),$$ where $p_\SSS\in\U_q(\g)$ is the element acting as the orthogonal projection onto the space of $U_q(\kk_\SSS)$-invariant vectors.
\end{Rem}

For some weights $\lambda\in P^+$ the vector $h_{w^{-1}\lambda}$ is already of lowest weight in~$V_\lambda$, so that $v_\lambda=\bar h_{w^{-1}\lambda}\otimes h_{w^{-1}\lambda}$. Namely, this happens if and only if $w^{-1}\lambda$ is fixed by~$W_\SSS$, or equivalently, $w_0\lambda$ is fixed by $W_\SSS$. To describe a class of such weights, define an involution on $P$ by \begin{equation}\label{EqInv} \bar\lambda=-w_0\lambda.\end{equation} 
Denote by $P(\SSS^c)$ the subgroup of $P$ generated by the fundamental weights~$\omega_r$ corresponding to $\alpha_r\in\SSS^c=\Phi^+\setminus\SSS$, and put $P^+(\SSS^c)=P(\SSS^c)\cap P^+$. Then all weights in $P(\SSS^c)$ are fixed by $W_\SSS$. Therefore
$$
k_{-4\lambda}=U(h_\lambda,h_{w^{-1}\lambda})^*U(h_\lambda,h_{w^{-1}\lambda})
=U(h_\lambda,h_{w_0\lambda})^*U(h_\lambda,h_{w_0\lambda}) \ \ \text{if}\ \ \bar\lambda\in P^+(\SSS^c).
$$
The elements $U(h_\lambda,h_{w_0\lambda})^*U(h_\lambda,h_{w_0\lambda})$ with $\bar\lambda\in P(\SSS^c)$ are known to generate $\Pol(\G_q/\K_{\SSS,q})$ as a right $U_q(\g)$-module \cite[Theorem~2.5]{S}.
We thus get the following. 

\begin{Prop}\label{pgen}
The elements $k_{-4\lambda}$, $\bar\lambda\in P^+(\SSS^c)$, generate $\Pol(\G_q/\K_{\SSS,q})$ as a right $U_q(\g)$-module.
\end{Prop}

Let us now establish some basic commutation relations for the elements $k_{-4\lambda}$.

\begin{Lem}\label{lkbasic}
We have:
\begin{itemize}
\item[{\rm 1)}] $k_{-4\lambda}k_{-4\mu}=k_{-4(\lambda+\mu)}$ for all $\lambda,\mu\in P^+$;
\item[{\rm 2)}] if $\xi$ is a weight vector in a type $I$ $U_q(\g)$-module $V$ and $\eta\in V$ is a $U_q(\kk_\SSS)$-invariant vector, then
$$
k_{-4\lambda}U(\xi,\eta)=q^{2(\lambda,\wt(\xi))}U(\xi,\eta)k_{-4\lambda}.
$$
\end{itemize}
\end{Lem}

\begin{proof}
Part (1) follows from the proof of Proposition~\ref{psoibelman}(2). Alternatively, this follows from the corresponding result for the elements $U(h_\lambda,h_{w^{-1}\lambda})^*U(h_\lambda,h_{w^{-1}\lambda}),$ which in turn can be proved analogously to (or deduced from) \cite[9.1.10]{Jo}.

Part (2) follows from Corollary~\ref{ccommut}, since $\eta$ is necessarily of weight zero and $U(\xi,\eta)\in\Pol(\G_q/\K_{\SSS,q})$.
\end{proof}

For $\lambda\in P^+$ put $k_{-\lambda}=k_{-4\lambda}^{1/4}\in C(\G_q/\K_{\SSS,q})$, where $C(\G_q/\K_{\SSS,q})$ denotes the universal C$^*$-envelope of $\Pol(\G_q)$.
\footnote{Although the necessity of taking fourth roots is a peculiarity of our conventions, any choice of convention would at least require taking of square roots.} In view of part (1) of the above lemma this definition is unambiguous. We can then consider the $^*$-algebra  obtained by adjoining to $\Pol(\G_q/\K_{\SSS,q})$ the elements $k_{-\lambda}$ and their (formal) inverses $k_{\lambda}$. By multiplicativity, we can extend the definition of $k_\lambda$ to all $\lambda\in P$.

To obtain a concrete realization of this algebra, let $V_{\SSS}\subseteq \Hsp_{w}=\ell^2(\N)^{\otimes l(w)}$ be the algebraic subspace spanned by the standard basis of the Hilbert space $\ell^2(\N)^{\otimes l(w)}$, or equivalently and more canonically, by the joint eigenvectors of the operators $\theta_w(k_{-4\lambda})$ for $\lambda\in P^+$. Then we may identify $\Pol(\G_q/\K_{\SSS,q})$ with a subalgebra of $\End(V_{\SSS})$ by means of~$\theta_w$. The operators $k_\lambda$, $\lambda\in P$, are well-defined  and invertible on $V_\SSS$, so we can make the following definition.

\begin{Def}
We denote by $\Pol(\G_q/\K_{\SSS,q})_{\ext}$ the subalgebra of $\End(V_\SSS)$ generated by $\Pol(\G_q/\K_{\SSS,q})$ and the self-adjoint operators $k_{\lambda}$, $\lambda\in P$.
\end{Def}

The $^*$-algebra $\Pol(\G_q/\K_{\SSS,q})_{\ext}$ can be thought of as an algebra of functions on the big Schubert cell of $\G_q/\K_{\SSS,q}$, considered as a real manifold. 

As we will see later, the following elements together with $k_\lambda$, $\lambda\in P$, form generators of $\Pol(\G_q/\K_{\SSS,q})_{\ext}$.

\begin{Def} For $r\in I$, we define \[x_r^+ = (q_r^{-1}-q_r)^{-1} (k_{-4\omega_r} \lhd E_r)k_{4\omega_r-\alpha_r} \in \Pol(\G_q/\K_{\SSS,q})_{\ext},\]  and then define $x_r^- = (x_r^+)^*$.
\end{Def}

To get an explicit expression, for $\lambda\in P^+$ we compute:
\begin{align*}
k_{-4\lambda}\lhd E_r&=q^{-(\rho-w\rho,\lambda)}U(\bar h_\lambda\otimes h_\lambda,v_\lambda)\lhd E_r\\
&=q^{-(\rho-w\rho,\lambda)}U(F_r(\bar h_\lambda\otimes h_\lambda),v_\lambda)\\
&=q^{-(\rho-w\rho,\lambda)+(\alpha_r,\lambda)/2}U(\bar h_\lambda\otimes F_r h_\lambda,v_\lambda).
\end{align*}
By Proposition~\ref{psoibelman}(1) we therefore get
\begin{align*}
k_{-4\lambda}\lhd E_r&=q^{-(\rho-w\rho,\lambda)+(\alpha_r,\lambda)/2}U(\bar h_\lambda\otimes F_r h_\lambda,\bar h_{w^{-1}\lambda}\otimes h_{w^{-1}\lambda}) \mod\ker\theta_w\\
&=q_r^{(\alpha_r^\vee,\lambda)/2}U(h_\lambda,h_{w^{-1}\lambda})^*
U(F_rh_\lambda,h_{w^{-1}\lambda})\mod\ker\theta_w.
\end{align*}
Thus, letting $\lambda=\omega_r$, we get
\begin{equation} \label{explus}
x_r^+=q_r^{1/2}(q_r^{-1}-q_r)^{-1}\theta_w(U(h_{\omega_r},h_{w^{-1}\omega_r})^*
U(F_rh_{\omega_r},h_{w^{-1}\omega_r}))k_{4\omega_r-\alpha_r}.
\end{equation}

\subsection{Degenerate quantized universal enveloping algebra}

Put $\varepsilon_r = 1$ if $\bar\alpha_r\in \SSS$ and $\varepsilon_r=0$ if $\bar\alpha_r\notin \SSS$ (cf. \eqref{EqInv}). 

\begin{Def}
Denote by $U_q(\g;\SSS)$ %
the unital algebra with the universal property that it contains and is generated by $U_q(\bb^+)$ and $U_q(\bb^-)$ as subalgebras, and with $L_{\omega}^{-1} = L_{\omega}'$ and \[\lbrack E_r,F_s\rbrack = \delta_{rs} \frac{\varepsilon_{r}L_{\alpha_r}^2-L_{\alpha_r}^{-2}}{q_r-q_r^{-1}}.\]
\end{Def}

The algebra $U_q(\g;\SSS)$ is a particular example of algebras studied in \cite{DeC1}. It can be seen as a degeneration of $U_q(\g)$. Indeed, let $b\colon P\rightarrow \R^*_+$ be a (multiplicative) character on the weight lattice. Then upon writing $E'_r = b_{\alpha_r}E_r$, $F'_r = b_{\alpha_r}F_r$ and $L_{\beta}'=b_{\beta}^{-1}L_{\beta}$ in $U_q(\g)$, we see that \[\lbrack E_r',F_r'\rbrack = \frac{b_{\alpha_r}^4L_{\alpha_r'}'^{2}-L_{\alpha_r}'^{-2}}{q_r-q_r^{-1}},\] while the other commutation relations remain unchanged. Letting the appropriate $b_{\alpha_r}$ tend to zero, we obtain the commutation relations for $U_q(\g;\SSS)$.

Note that the above rescaling is invariant with respect to the adjoint action $\lhd$. It follows that there is a limit action of $U_q(\g)$ on $U_q(\g;\SSS)$, see~\cite{DeC1}. Namely, this action is characterized by the property that if $x\in U_q(\g;\SSS)$ and $y\in U_q(\bb^\pm)\subseteq U_q(\g)$, then $x\lhd y=S(y_{(1)})xy_{(2)}$.

We can now formulate our main result.

\begin{Theorem}\label{TheoHom} There is a surjective unital $^*$-homomorphism  \[\Psi\colon U_q(\g;\SSS)\rightarrow \Pol(\G_q/\K_{\SSS,q})_{\ext}\] such that $\Psi(L_{\omega}) = k_{\omega}$, $\Psi(E_r) = x_r^+$ and $\Psi(F_r)=x_r^-$ for all $\omega\in P$ and $r\in I$. This homomorphism has the following properties:
\begin{itemize}
\item[{\rm 1)}] the right action of $U_q(\g)$ on $\Pol(\G_q/\K_{\SSS,q})$ extends uniquely to an action on the algebra $\Pol(\G_q/\K_{\SSS,q})_{\ext}$ such that the homomorphism $\Psi$ becomes $U_q(\g)$-equivariant;
\item[{\rm 2)}] the subalgebra $U_q(\g;\SSS)_{\fin}\subseteq U_q(\g;\SSS)$ of locally finite vectors with respect to the action of $U_q(\g)$ is mapped onto $\Pol(\G_q/\K_{\SSS,q})$.
\end{itemize}
\end{Theorem}

To put this differently, if we denote by $V_\SSS$ the unique irreducible highest weight $U_q(\g;\SSS)$-module with highest weight $0$ (see~\cite{DeC1}), then the image of $U_q(\g;\SSS)_{\fin}$ in $\End(V_\SSS)$ is $U_q(\g)$-equiva\-riantly isomorphic to $\Pol(\G_q/\K_{\SSS,q})$.

\subsection{Proof of the theorem}

We will prove the theorem in a series of steps. It will be convenient to use the following definition.

\begin{Def}
We say that an operator $a\in\End(V_\SSS)$ is of weight $\omega$ if $k_\lambda a k_{-\lambda}=q^{(\lambda,\omega)/2}a$ for all $\lambda\in P$.
\end{Def}

Then an equivalent formulation of Lemma~\ref{lkbasic}(2) is that if $a\in \Pol(\G_q/\K_{\SSS,q})$ is  such that $a\lhd L_\lambda=q^{-(\lambda,\omega)/2}a$ for all $\lambda\in P$, then $a$ is of weight $\omega$.

\begin{Lem}\label{LemCome} For all $a\in \Pol(\G_q/\K_{\SSS,q})$, we have \begin{equation}\label{EqEr} a\lhd E_r = -q_rx_r^+ak_{\alpha_r} + k_{\alpha_r}ax_r^+.\end{equation} Similarly, \begin{equation}\label{EqFr} a\lhd F_r = -q_r^{-1}x_r^-ak_{\alpha_r} + k_{\alpha_r}ax_r^-.\end{equation}
\end{Lem}
\begin{proof} As $(x\lhd y)^* = x^* \lhd S(y)^*$ for all $y\in U_q(\g)$, and as the elements $k_{\alpha_r}$ are self-adjoint, it is sufficient to prove \eqref{EqEr}.

Let $a\in \Pol(\G_q/\K_{\SSS,q})$ be  such that $a\lhd L_\lambda=q^{-(\lambda,\omega)/2}a$ for all $\lambda\in P$, so that $a$ is of weight~$\omega$. Then $a\lhd E_r$ is of weight $\omega+\alpha_r$. In particular, $x_r^+$ is of weight $\alpha_r$. Recalling the definition of $x^+_r$ we conclude, upon multiplying by $k_{-4\omega_r}$ on the right, that \eqref{EqEr} is equivalent to
\[-q^{\frac{1}{2}(4\omega_r-\alpha_r,\omega)}(k_{-4\omega_r}\lhd E_r)a + q^{\frac{1}{2}(\alpha_r,\omega)}a(k_{-4\omega_r}\lhd E_r) = (q_r^{-2}-1)(a\lhd E_r)k_{-4\omega_r}.\] But this identity follows from $(k_{-4\omega_r}a-q^{-2(\omega_r,\omega)}ak_{-4\omega_r})\lhd E_r=0$ using that $(bc)\lhd x = (b\lhd x_{(1)})(c\lhd x_{(2)})$.
\end{proof}

\begin{Lem} \label{LemHom} There is a unital homomorphism \[\Psi\colon U_q(\bb)\rightarrow \Pol(\G_q/\K_{\SSS,q})_{\ext}\] such that $\Psi(L_{\omega}) = k_{\omega}$ and $\Psi(E_r) = x_r^+$ for all $\omega\in P$ and $r\in I$.
\end{Lem}

\begin{proof} Since $x^+_r$ is of weight $\alpha_r$, it is clear that $\Psi$ is well-defined as a homomorphism of~$\tilde{U}_q(\bb)$. It follows that we can define a right action of $\tilde{U}_q(\bb)$ on $\Pol(\G_q/\K_{\SSS,q})_{\ext}$ by \[x \blacktriangleleft y = \Psi(S(y_{(1)}))x\Psi(y_{(2)}).\]

Write $\pi\colon  \tilde{U}_q(\bb) \rightarrow U_q(\bb)$ for the quotient map. From Lemma~\ref{LemCome} we easily infer that \[x\lhd \pi(y) = x \blacktriangleleft y\] for $x\in \Pol(\G_q/\K_{\SSS,q})$ and $y\in \{L_{\omega},E_r\}\subseteq \tilde{U}_q(\bb)$. It then follows that this identity holds for arbitrary $y$. In particular, by definition~\eqref{ek} of $k_{-4\omega_s}$ we have that for $r\neq s$, up to scalar factors, \begin{eqnarray*}
\Psi(L_{-4\omega_s}\lhd E_sE_r^{1-a_{rs}}) &=& k_{-4\omega_s}\lhd E_sE_r^{1-a_{rs}}\\&=& U\left(F_r^{1-a_{rs}}F_s(\bar h_{\omega_s}\otimes h_{\omega_s}),v_{\omega_s}\right).
\end{eqnarray*}
But $F_r^{1-a_{rs}}F_s(\bar h_{\omega_s}\otimes h_{\omega_s})$ is a multiple of $\bar h_{\omega_s}\otimes (F_r^{1-a_{rs}}F_sh_{\omega_s})$. As $F_rh_{\omega_s}=0$, we have \[F_r^{1-a_{rs}}F_sh_{\omega_s} = (-1)^{1-a_{rs}}  \sum_{k=0}^{1-a_{rs}}  (-1)^k \bqn{1-a_{rs}}{k}{r} F_r^kF_sF_r^{1-a_{rs}-k}h_{\omega_s} = 0.\] Hence $\Psi$ preserves the Serre relations and descends to $U_q(\bb)$.
\end{proof}

Since $x_r^-=(x_r^+)^*$, this also gives the following. 

\begin{Cor} There is a unital homomorphism \[\Psi\colon  U_q(\bb^-)\rightarrow \Pol(\G_q/\K_{\SSS,q})_{\ext}\] such that $\Psi(L'_{\omega}) = k_{-\omega}$ and $\Psi(F_r) = x_r^{-}$.
\end{Cor}

To show that $\Psi\colon U_q(\bb^\pm)\to\Pol(\G_q/\K_{\SSS,q})_{\ext}$ defines a homomorphism of $U_q(\g;\SSS)$ it remains to find the commutation relations between the elements $x^+_r$ and $x_s^-$.

\begin{Lem} For $r\neq s$, we have $\lbrack x_r^+,x_s^-\rbrack =0$.\end{Lem}
\begin{proof} For $r\neq s$, since $E_s(\bar h_{\omega_r}\otimes h_{\omega_r})=0$, we have $k_{-4\omega_r}\lhd F_s=0$. Hence $$(k_{-4\omega_r}\lhd E_r)\lhd F_s = k_{-4\omega_r}\lhd (F_sE_r) = 0.$$ Writing out the left hand side by means of the definition of $x_r^+$ and Lemma \ref{LemCome}, we arrive at the commutation of $x_r^+$ with $x_s^-$.
\end{proof}

The case $r=s$ is more complicated. We start with the following.

\begin{Lem}\label{LemElCent} The element \begin{equation}\label{EqElCent} k_{\alpha_r}^{-2} \lbrack x_r^+,x_r^- \rbrack  + (q_r-q_r^{-1})^{-1}k_{\alpha_r}^{-4}\end{equation} is central in $\Pol(\G_q/\K_{\SSS,q})_{\ext}$.\end{Lem}
\begin{proof} Take $x \in \Pol(\G_q/\K_{\SSS,q})$. Then \[(x\lhd E_r)\lhd F_r - (x\lhd F_r) \lhd E_r = x\lhd \left(\frac{L_{\alpha_r}^2-L_{\alpha_r}^{-2}}{q_r-q_r^{-1}}\right).\] Writing this out using Lemma~\ref{LemCome} and the known commutations between the elements $k_{\omega}$ and $x_r^+$ and $x_r^-$, we obtain that \eqref{EqElCent} commutes with all $x\in \Pol(\G_q/\K_{\SSS,q})$. As \eqref{EqElCent} clearly commutes with all $k_{\omega}$, we have in fact that \eqref{EqElCent} is central in $\Pol(\G_q/\K_{\SSS,q})_{\ext}$.
\end{proof}

\begin{Cor} \label{cepsexistence}
There exist $\varepsilon_r\in \R$ such that \[\lbrack x_r^+,x_r^-\rbrack = \frac{\varepsilon_r k_{\alpha_r}^2 - k_{\alpha_r}^{-2}}{q_r-q_r^{-1}}.\]
\end{Cor}
\begin{proof} Write \[z = k_{-8\omega_r+4\alpha_r}(k_{\alpha_r}^{-2} \lbrack x_r^+,x_r^- \rbrack  + (q_r-q_r^{-1})^{-1}k_{\alpha_r}^{-4}).\] Then $z\in \Pol(\G_q/\K_{\SSS,q})$, and by Lemma \ref{LemElCent}, we have \[zxk_{-8\omega_r+4\alpha_r}  = k_{-8\omega_r+4\alpha_r}xz\] for all $x\in \Pol(\G_q/\K_{\SSS,q})$. Note also that $k_{-8\omega_r+4\alpha_r}\in \Pol(\G_q/\K_{\SSS,q})$. As $\theta_w$ is an irreducible representation of $\Pol(\G_q/\K_{\SSS,q})$, we deduce that $z$ is a scalar multiple of $k_{-8\omega_r+4\alpha_r}$. This implies that there exist $\varepsilon_r\in \C$ as in the statement of the corollary. As the left hand side is clearly self-adjoint, it follows that $\varepsilon_r\in \R$.
\end{proof}

It remains to show that $\varepsilon_r = 1$ if $\bar\alpha_r\in \SSS$ and $\varepsilon_r=0$ if $\bar\alpha_r\notin \SSS$.

Before turning to the proof, let us rewrite Corollary~\ref{cepsexistence}.
By~\eqref{explus} we have
\[[x_r^+,x_r^-]=q_r(q_r^{-1}-q_r)^{-2}(A_r - B_r),\] where
\begin{eqnarray*}
A_r &=& \theta_w(U(h_{\omega_r},h_{w^{-1}\omega_r})^*U(F_rh_{\omega_r},h_{w^{-1}\omega_r}))k_{8\omega_r-2\alpha_r}\\ && \qquad \times
\theta_w(U(F_rh_{\omega_r},h_{w^{-1}\omega_r})^*U(h_{\omega_r},h_{w^{-1}\omega_r})),\\
B_r &=& k_{4\omega_r-\alpha_r}\theta_w(U(F_rh_{\omega_r},h_{w^{-1}\omega_r})^*U(h_{\omega_r},h_{w^{-1}\omega_r})\\ && \qquad \times U(h_{\omega_r},h_{w^{-1}\omega_r})^*
U(F_rh_{\omega_r},h_{w^{-1}\omega_r}))k_{4\omega_r-\alpha_r}.
\end{eqnarray*}
By Corollary~\ref{ccommut}, for any $m\ge0$, the operator $\theta_w(U(F_r^mh_\lambda,h_{w^{-1}\lambda}))$ is of weight $m\alpha_r$. Furthermore, by the same corollary we know that the operators  $\theta_w(U(h_\lambda,h_{w^{-1}\lambda}))$ and $\theta_w(U(h_\lambda,h_{w^{-1}\lambda})^*)$ for $\lambda\in P^+$ behave as $k_{-2\lambda}$ with respect to commutation relations with $\theta_w(U(F_r^mh_\lambda,h_{w^{-1}\lambda}))$ (in fact, they even coincide with $k_{-2\lambda}$ up to phase factors by the proof of Proposition~\ref{psoibelman}(2)).
 Using this we get
\begin{multline*}
[x_r^+,x_r^-]=q_r^{-1}(q_r^{-1}-q_r)^{-2}\big[\theta_w(U(F_rh_{\omega_r},h_{w^{-1}\omega_r})
U(F_rh_{\omega_r},h_{w^{-1}\omega_r})^*)\\
-\theta_w(U(F_rh_{\omega_r},h_{w^{-1}\omega_r})^*
U(F_rh_{\omega_r},h_{w^{-1}\omega_r}))\big]k_{4\omega_r-2\alpha_r}.
\end{multline*}
By \eqref{estar} we also have
$$
U(F_rh_{\omega_r},h_{w^{-1}\omega_r})^*=q_rq^{-(\rho,\omega_r-w^{-1}\omega_r)}
U(\overline{F_rh_{\omega_r}},\bar h_{w^{-1}\omega_r}).
$$
Thus Corollary~\ref{cepsexistence} reads as
\begin{multline} \label{eeps}
q^{-(\rho,\omega_r-w^{-1}\omega_r)}(q_r-q_r^{-1})^{-1}\theta_w\big[ U(F_rh_{\omega_r}\otimes\overline{F_rh_{\omega_r}},h_{w^{-1}\omega_r}\otimes \bar h_{w^{-1}\omega_r})\\
-U(\overline{F_rh_{\omega_r}}\otimes F_rh_{\omega_r},\bar h_{w^{-1}\omega_r}\otimes h_{w^{-1}\omega_r})\big]
=\varepsilon_r k_{-4\omega_r+4\alpha_r}-k_{-4\omega_r}.
\end{multline}

We now apply Lemma~\ref{LemSwitch}(1) to the first summand on the left hand side. Using properties \eqref{ermatrix} of the $R$-matrix and the identities
$$
E_rF_rh_{\omega_r}=h_{\omega_r},\ \ F_r\overline{F_rh_{\omega_r}}=-\overline{E_rF_rh_{\omega_r}}=-\bar h_{\omega_r},
$$
we get
$$
\RR_{21}(\overline{F_rh_{\omega_r}}\otimes F_rh_{\omega_r})=q^{-(\omega_r,\omega_r)}\overline{F_rh_{\omega_r}}\otimes F_rh_{\omega_r}-q^{-(\omega_r,\omega_r)}(q_r-q_r^{-1})\bar h_{\omega_r}\otimes h_{\omega_r}
$$
and
$$
\RR^{-1}(\bar h_{w^{-1}\omega_r}\otimes h_{w^{-1}\omega_r})=q^{(\omega_r,\omega_r)}\tilde\RR^{-1}(\bar h_{w^{-1}\omega_r}\otimes h_{w^{-1}\omega_r}).
$$
Therefore \eqref{eeps} becomes
\begin{multline*}
q^{-(\rho,\omega_r-w^{-1}\omega_r)}(q_r-q_r^{-1})^{-1}\theta_w(
U(\overline{F_rh_{\omega_r}}\otimes F_rh_{\omega_r},(\tilde\RR^{-1}-1)(\bar h_{w^{-1}\omega_r}\otimes h_{w^{-1}\omega_r})))\\
-q^{-(\rho,\omega_r-w^{-1}\omega_r)}\theta_w(
U(\bar h_{\omega_r}\otimes h_{\omega_r},\tilde\RR^{-1}(\bar h_{w^{-1}\omega_r}\otimes h_{w^{-1}\omega_r})))
=\varepsilon_r k_{-4\omega_r+4\alpha_r}-k_{-4\omega_r}.
\end{multline*}
In view of Proposition~\ref{psoibelman}(1) the second term on the left hand side is exactly $-k_{-4\omega_r}$. Thus we finally get that
\begin{multline}\label{eeps1}
q^{-(\rho,\omega_r-w^{-1}\omega_r)}(q_r-q_r^{-1})^{-1}\theta_w(
U(\overline{F_rh_{\omega_r}}\otimes F_rh_{\omega_r},(\tilde\RR^{-1}-1)(\bar h_{w^{-1}\omega_r}\otimes h_{w^{-1}\omega_r})))\\ =\varepsilon_r k_{-4\omega_r+4\alpha_r}.
\end{multline}

\begin{Lem}
With $\varepsilon_r$ as in Corollary~\ref{cepsexistence}, we have $\varepsilon_r = 1$ if $\bar\alpha_r\in \SSS$ and $\varepsilon_r=0$ if $\bar\alpha_r\notin \SSS$.
\end{Lem}

\bp
Define an involution on $I$ by $\alpha_{\bar r}=\bar\alpha_r$.

\smallskip

Assume $\bar \alpha_r\not\in\SSS$. In this case the equality $\varepsilon_r=0$ follows immediately from \eqref{eeps1}, since $\bar\omega_r=\omega_{\bar r}\in P(\SSS^c)$ and so $w^{-1}\omega_r$ is already the lowest weight of~$V_{\omega_r}$, whence $$(\tilde\RR^{-1}-1)(\bar h_{w^{-1}\omega_r}\otimes h_{w^{-1}\omega_r})=0.$$

\smallskip

Turning to the case $\bar\alpha_r\in\SSS$, assume first that $\SSS=\{\bar\alpha_r\}$. Then $w=w_0s_{\bar r}=s_rw_0$ and
$w^{-1}\omega_r=s_{\bar r}w_0\omega_r=-s_{\bar r}\omega_{\bar r}=-\omega_{\bar r}+\alpha_{\bar r}$. In particular, up to a phase factor, $h_{w^{-1}\omega_r}$ equals $E_{\bar r}h_{w_0\omega_r}$. It follows that
$$
(\tilde\RR^{-1}-1)(\bar h_{w^{-1}\omega_r}\otimes h_{w^{-1}\omega_r})
=(q_{\bar r}-q_{\bar r}^{-1})\bar h_{w_0\omega_r}\otimes h_{w_0\omega_r}.
$$
Then, since $F_rh_{\omega_r}$ coincides with $h_{s_r\omega_r}$ up to a phase factor, \eqref{eeps1} becomes
$$
\theta_{s_rw_0}(U(h_{s_r\omega_r},h_{w_0\omega_r})^*U(h_{s_r\omega_r},h_{w_0\omega_r})) =\varepsilon_r k_{-4\omega_r+4\alpha_r}.
$$
By results of \cite{Soi1} (we used similar arguments in the proof of Proposition~\ref{psoibelman}(2)), the operator on the left hand side can be explicitly computed. In particular, we know that $e_0^{\otimes (l(w_0)-1)}$ is an eigenvector of this operator with eigenvalue $1$. But the same is true for the operator $k_{-4\omega_r+4\alpha_r}$. Hence $\varepsilon_r=1$.

\smallskip

Consider now the case of an arbitrary $\SSS $ containing $\bar\alpha_r$. Denote by $v_w$ the vector $e_0^{\otimes l(w)}\in\Hsp_w$. Since $x_r^+$ is of weight $\alpha_r$, so that $k_{-\lambda} x_r^+=q^{-(\lambda,\alpha_r)/2}x_r^+k_{-\lambda}$, and $v_w$ is an eigenvector of $k_{-\lambda}$ with the largest eigenvalue $1$ for $\lambda\in P^+$, we have $x_r^+v_w=0$. This together with Corollary~\ref{cepsexistence} implies that the equality $\varepsilon_r=1$ is equivalent to $x_r^-v_w=0$, which by \eqref{explus} is, in turn, equivalent to
\begin{equation}\label{eeps3}
\theta_w(U(F_rh_{\omega_r},h_{w^{-1}\omega_r}))^*v_w=0.
\end{equation}
We want to show that this follows from the case $\SSS=\{\bar\alpha_r\}$ considered above, which implies that
\begin{equation} \label{eeps4}
\theta_{s_rw_0}(U(F_rh_{\omega_r},h_{s_{\bar r}w_0\omega_r}))^*v_{s_rw_0}=0.
\end{equation}

Recall that $w_{\SSS,0}$ denotes the longest element in $W_\SSS$. Put $u=w_{\SSS,0}s_{\bar r}$, so that $s_rw_0=wu$ and $l(s_rw_0)=l(w)+l(u)$. Choosing an orthonormal basis $\{\xi_j\}_j$ in $V_{\omega_r}$, we have
$$
\theta_{s_rw_0}(U(F_rh_{\omega_r},h_{s_{\bar r}w_0\omega_r}))^*v_{s_rw_0}
=\sum_j \theta_{w}(U(F_rh_{\omega_r},\xi_j))^*v_{w}\otimes\theta_u(U(\xi_j,h_{s_{\bar r}w_0\omega_r}))^*v_{u}.
$$
Therefore in order to deduce \eqref{eeps3} from \eqref{eeps4} it suffices to show that $$(\theta_u(U(h_{w^{-1}\omega_r},h_{s_{\bar r}w_0\omega_r}))^*v_u,v_u)\ne0$$ and $(\theta_u(U(\xi,h_{s_{\bar r}w_0\omega_r}))^*v_u,v_u)=0$ for $\xi$ orthogonal to $h_{w^{-1}\omega_r}$.
As we already discussed after Definition~\ref{dk}, the vector $h_{w^{-1}\omega_r}$ is a highest weight vector of the $U_q(\kk_\SSS)$-module $V=U_q(\kk_\SSS)h_{w_0\omega_r}$. Since we also have $h_{s_{\bar r}w_0\omega_r}\in V$ and $\theta_u$ is defined using the homomorphism $\Pol(\G_q)\to\Pol(\K_{\SSS,q})$, we thus see that the problem at hand is entirely about $\K_{\SSS,q}$. So without loss of generality we may assume that $\SSS=\Phi^+$ (of course, our original problem of computing $\varepsilon_r$ is trivial in this case).

Thus we have to show that $(\theta_{w_0s_{\bar{r}}}(U(h_{\omega_r},h_{s_{\bar r}w_0\omega_r}))^*v_{w_0s_{\bar{r}}},v_{w_0s_{\bar{r}}})\ne0$ and $$(\theta_{\omega_0s_{\bar{r}}}(U(\xi,h_{s_{\bar r}w_0\omega_r}))^*v_{\omega_0s_{\bar{r}}},v_{\omega_0s_{\bar{r}}})=0$$ for $\xi$ orthogonal to $h_{\omega_r}$. The first claim follows from Proposition~\ref{psoibelman}(2). For the second, consider a weight vector $\xi\in V_{\omega_r}$ with $\wt(\xi)\ne\omega_r$. Then by Corollary~\ref{ccommut} the operator $\theta_{\omega_0s_{\bar{r}}}(U(\xi,h_{s_{\bar r}w_0\omega_r}))^*$ is of weight $\wt(\xi)-\omega_r\ne0$ with respect to the operators $k_\lambda$ defined for $\SSS=\{\alpha_{\bar r}\}$. Hence $\theta_{\omega_0s_{\bar{r}}}(U(\xi,h_{s_{\bar r}w_0\omega_r}))^*$ maps any joint eigenvector of $k_\lambda$, $\lambda\in P$, into an orthogonal vector. This proves the second claim.
\end{proof}

This lemma finishes the proof of the existence of a unital homomorphism $$\Psi\colon U_q(\g;\SSS)\to\Pol(\G_q/\K_{\SSS,q})_{\ext}$$ such that $\Psi(L_\omega)=k_\omega$, $\Psi(E_r)=x_r^+$, $\Psi(F_r)=x_r^-$. Clearly, this homomorphism is $^*$-preserving.

Lemma~\ref{LemCome} and the relations in $U_q(\g;\SSS)$ imply that the right action of $U_q(\g)$ on the algebra $\Pol(\G_q/\K_{\SSS,q})$ extends to an action on $\Pol(\G_q/\K_{\SSS,q})_{\ext}$ defined by
$$
a\lhd L_\omega=k_{-\omega}ak_\omega,\ \ a\lhd E_r = -q_rx_r^+ak_{\alpha_r} + k_{\alpha_r}ax_r^+,\ \ a\lhd F_r = -q_r^{-1}x_r^-ak_{\alpha_r} + k_{\alpha_r}ax_r^-.
$$
In fact, this is basically the content of Lemmas~\ref{LemHom}--\ref{LemElCent}. With respect to this action the homomorphism $\Psi\colon U_q(\g;\SSS)\to\Pol(\G_q/\K_{\SSS,q})_{\ext}$ becomes $U_q(\g)$-equivariant.

By \cite[Proposition~3.2]{DeC1} the subalgebra  $U_q(\g;\SSS)_{\fin}\subseteq U_q(\g;\SSS)$ is generated by the elements $L_{-4\lambda}$, $\lambda\in P^+$, as a right $U_q(\g)$-module. Since by Proposition~\ref{pgen} the elements $k_{-4\lambda}$ generate $\Pol(\G_q/\K_{\SSS,q})$ as a right $U_q(\g)$-module, it follows that $\Psi(U_q(\g;\SSS)_{\fin})=\Pol(\G_q/\K_{\SSS,q})$. Since $\Pol(\G_q/\K_{\SSS,q})_{\ext}$ is generated by $\Pol(\G_q/\K_{\SSS,q})$ and the elements $k_\omega$, $\omega\in P$, we finally conclude that $\Psi\colon U_q(\g;\SSS)\to\Pol(\G_q/\K_{\SSS,q})_{\ext}$ is surjective. This completes the proof of Theorem~\ref{TheoHom}.

\end{document}